\documentclass[12pt,a4paper]{article}
\usepackage[active]{srcltx}
\usepackage{amsfonts}
\usepackage{amsmath}
\usepackage{amsbsy}
\usepackage{amsxtra}
\usepackage{latexsym}
\usepackage{amssymb}
\usepackage{url}
\usepackage{cite}
\usepackage{pstricks,pst-node}
\makeatletter
\g@addto@macro\thesection.
\makeatother

\newtheorem{theorem}{Theorem}
\newtheorem{lemma}{Lemma}
\newtheorem{corollary}{Corollary}
\newtheorem{proposition}{Proposition}
\newtheorem{remark}{Remark}

\DeclareMathOperator{\cone}{cone}
\DeclareMathOperator{\intr}{int}

\DeclareMathOperator{\spa}{sp}

\DeclareMathOperator{\diag}{diag}

\newcommand{\lng}{\langle}
\newcommand{\rng}{\rangle}

\newcommand{\R}{\mathbb R}
\newcommand{\N}{\mathbb N}
\newcommand{\p}{\partial}
\newcommand{\mc}{\mathcal}

\newenvironment{proof}{{\noindent\bf Proof.}}{\hfill$\Box$\\}

\begin{document}

\title{Order isotonicity of the metric projection onto a closed convex cone
\thanks{{\it 1991 A M S Subject Classification:} Primary 90C33,
Secondary 15A48; {\it Key words and phrases:} convex cones, isotone projections. }
}
\author{A. B. N\'emeth\\Faculty of Mathematics and Computer Science\\Babe\c s Bolyai University, Str. Kog\u alniceanu nr. 1-3\\RO-400084 Cluj-Napoca, Romania\\email: nemab@math.ubbcluj.ro \and S. Z. N\'emeth\\School of Mathematics, The University of Birmingham\\The Watson Building, Edgbaston\\Birmingham B15 2TT, United Kingdom\\email: nemeths@for.mat.bham.ac.uk}
\date{}
\maketitle

\begin{abstract}
The basic tool for solving problems in metric geometry and isotonic regression
is the metric projection onto closed convex cones. Isotonicity of these
projections with respect to a given order relation can facilitate finding the solutions of the above problems.
In the recent note \cite{NemethNemeth2016} this problem was studied for the coordinate-wise ordering. 
This study was the starting point for further investigations, such as the ones presented here.
The order relation in the Euclidean space endowed by a proper cone is considered
and the proper cones admitting isotone metric projections with respect to this order relation are investigated.

\end{abstract}

\section{Introduction}

Let $\R^m$ be the $m$-dimensional Euclidean space endowed with the standard inner 
product $\lng\cdot,\cdot\rng:\R^m\times\R^m\to\R$ 
and the Euclidean norm $\|\cdot\|$ together with the topology this scalar product defines.

Denote by $P_D$ 
the projection mapping onto a nonempty closed convex set $D\subset \R^m,$ 
that is the mapping which associates
to $x\in \R^m$ the unique nearest point $P_Dx$ of $x$ in $D$ \cite{Zarantonello1971}:

\[ P_Dx\in D,\;\; \textrm{and}\;\; \|x-P_Dx\|= \inf \{\|x-y\|: \;y\in D\}. \]

Given an order relation $\preceq$ in $\R^m$, the closed convex set is
called an \emph{isotone projection set} if from $x\preceq y,\;x,\,y\in \R^m$,
it follows that $P_Dx\preceq P_Dy$.

\emph{Due to the importance of the projection operator in applications, it is
desirable to get a user friendly order relation for which the class
of the isotone projection sets is as large as possible.}

In this regard the interest is focused onto the widely used \emph{vectorial ordering}, because of its natural
connection to the vector-space structure of the Euclidean space $\R^n$ . It is usually endowed by a
cone $K$ and is denoted by $\leq_K$. (See the detailed explanation for the terms which are not defined here in the next section.)

If $\preceq =\leq_K$ for some cone $K$, then the isotone projection set $D$ is
called \emph{$K$-isotone}.

The investigations concerning the isotonicity with respect to
the order relation induced by a cone of the metric projection
onto a convex set go back to the paper
\cite{IsacNemeth1986} of G. Isac and A. B. N\'emeth, where 
the isotone projection cones (i.e., generating pointed closed convex 
cones $K$ admitting a $\leq_K$ isotone projection onto themselves) are characterized. 
The same authors \cite{IsacNemeth1990} and S. J. Bernau
\cite{Bernau1991} considered the similar problem for Hilbert spaces. In these papers
and in the applications in 
\cite{IsacNemeth1990b}, \cite{IsacNemeth2008c},
\cite{Nemeth2009a} (for the problem of solving nonlinear complementarity problems) the
ordering is defined by isotone projection cones. 

The next step is to get the family of closed convex sets which admit isotonic projection
with respect to a given ordering. In $\R^m$ with a given Cartesian reference system and
the coordinate-wise order relation the problem was settled in
\cite{Isac1996},
 \cite{NishimuraOk2012}, \cite{NemethNemeth2013}. If the ordering is induced by the Lorentz cone,
or ice cream cone it was settled in \cite{NemethNemeth2013}. 
The machinery which permits advances in this direction is to reduce
the general problem to isotone projection onto subspaces. It was developed
in \cite{NemethNemeth2013} as well as in \cite{NemethNemeth2012a}.

In important applications to metric geometry \cite{Dattorro2005} and regression theory
\cite{Kruskal1964,BarlowBartholomewBremnerBrunk1972,RobertsonWrightDykstra1988,BestChakravarti1990,WuWoodroofeMentz2001,DeLeeuwHornikMair2009,ShivelySagerWalker2009} the convex sets onto which the metric projection is considered
are closed convex cones. The papers \cite{NemethNemeth2009} and \cite{GuyaderJegouNemethNemeth2014}
exploited the fact that the totally ordered isotonic regression cone is an isotone projection cone too.

However, it is a very strong condition for a cone to be an isotone projection one. We would expect that considering order relations endowed
by more general cones may be useful in applications. In this regard we
consider in the present note the following particular case of the problem emphasized
at the beginning of our introduction:

{\bf Problem:} \emph{Given a proper cone $K$ we seek another proper cone $L$
with the property that $P_K$ is $L$-isotone.}

In our recent note \cite{NemethNemeth2016} the family
of closed convex cones admitting isotonic metric projections
with respect to the coordinate-wise ordering was determined. Every isotonic regression
cone belongs to this class. These results serve as justification and starting point
for the other theoretical results contained in the present note.

Our investigations rely on the results in \cite{NemethNemeth2013} as well as
\cite{NemethNemeth2012a}.

\section{The used terminology} \label{termin}

We aspire to be in line with the standard terminology from convex geometry. 
(see e.g. \cite{Rockafellar1970}). 

The non-empty set $K\subset \R^m$ is called a \emph{convex cone} if
(i) $K+K\subset K$ and (ii) $tK\subset K,\;\forall \;t\in \R_+ =[0,+\infty)$. All the cones used 
in this paper are convex.
The convex cone $K$ is called \emph{pointed}, if $(-K)\cap K=\{0\}.$

The convex cone $K$ is called {\it generating} if $K-K=\R^m$.

A generating closed convex pointed cone is called \emph{proper cone}.
 
For any $x,y\in \R^m$, by the equivalence $x\leq_K y\Leftrightarrow y-x\in K$, the 
convex cone $K$ 
induces an {\it order relation} $\leq_K$ in $\R^m$, that is, a binary relation, which is 
reflexive and transitive. This order relation is {\it translation invariant} 
in the sense that $x\leq_K y$ implies $x+z\leq_K y+z$ for all $z\in \R^m$, and 
{\it scale invariant} in the sense that $x\leq_Ky$ implies $tx\leq_K ty$ for any $t\in \R_+$.
If $\leq$ is a translation invariant and scale invariant order relation on $\R^m$, then 
$\leq=\leq_K$ with $K=\{x\in\R^m:0\leq x\}.$ The vector space $\R^m$
endowed with the relation $\leq_K$ is denoted by $(\R^m,K)$ and is called an \emph{ordered
Euclidean vector space}. In accordance, $\leq_K$ is called a \emph{vectorial ordering}. If $K$ is pointed, then $\leq_K$ is 
\emph{antisymmetric} too, that is $x\leq_K y$ and $y\leq_K x$ imply that $x=y.$

The set
$$ K= \cone\{x_1,\dots,x_m\}:=\{t^1x_1+\dots+t^m x_m:\;t^i\in \R_+,\;i=1,\dots,m\}$$
with $x_1,\,\dots,\,x_m$ linearly independent vectors is called a \emph{simplicial cone}.
A simplicial cone is proper.

The \emph{dual} of the convex cone $K$ is the set
$$K^*:=\{y\in \R^n:\;\lng x,y\rng \geq 0,\;\forall \;x\in K\}.$$
The dual of a convex cone is a closed convex cone.

A convex cone $K$ is called \emph{subdual} if $K\subset K^*$ and
it is called \emph{self-dual}, if $K=K^*.$ If $K$
is self-dual, then it is proper.

Suppose that $\R^m$ is endowed with a
Cartesian system. Let
$x,y\in \R^m$, $x=(x^1,...,x^m)$, $y=(y^1,...,y^m)$, where $x^i$, $y^i$ are the coordinates of
$x$ and $y$, respectively with respect to the Cartesian system. Then, the scalar product of $x$
and $y$ is the sum
$\lng x,y\rng =\sum_{i=1}^m x^iy^i.$

The set
\[\R^m_+=\{x=(x^1,...,x^m)\in \R^m:\; x^i\geq 0,\;i=1,...,m\}\]
is called the \emph{nonnegative orthant} of the above introduced Cartesian
system. It is a simplicial cone. A direct verification shows that $\R^m_+$ is a
self-dual cone.

Taking a Cartesian system in $\R^m$ and using the above introduced
notations,
the 
\emph{coordinatewise order}  $\leq$ in $\R^m$ is defined by
\[x=(x^1,...,x^m)\leq y=(y^1,...,y^m)\;\Leftrightarrow\;x^i\leq y^i,\;i=1,...,m.\]
By using the notion of the order relation induced by a cone, defined above, 
it is easy to see that $\leq =\leq_{\R^m_+}$.

A \emph{hyperplane} (through $a\in\R^m$) is a set of form
\begin{equation}\label{hyperplane}
	H(u,a)=\{x\in \R^m:\;\lng u,x\rng =\lng u,a\rng\},\;\;u\not= 0.
\end{equation}

A hyperplane $H(u,a)$ determines two \emph{closed halfspaces} $H_-(a,u)$ and
$H_+(u,a)$  of $\R^m$, defined by

\[H_-(u,a)=\{x\in \R^m:\; \lng u,x\rng \leq \lng u,a\rng\},\]
and
\[H_+(u,a)=\{x\in \R^m:\; \lng u,x\rng \geq \lng u,a\rng\}.\]

The hyperplane $H(u,0)$ is a \emph{supporting hyperplane} to the cone
$K$ if $K\subset H_-(u,0)$.

The proper cone $K$ is said \emph{strictly convex} if the dimension
$\dim(K\cap H(u,0))$ is at most $1$ for each supporting hyperplane of $K$.
The strictly convex proper cone $K$ is called also \emph{smooth} if
through each its boundary point $x\not=0$ there exist exactly one
supporting hyperplane to $K$.

The following auxiliary results are consequences of standard
reasonings in convex geometry (see e. g. \cite{Rockafellar1970} and \cite{Zarantonello1971}).

\begin{lemma}\label{intersec}
Let $K$ be a strictly convex proper cone and $L$ be a proper cone.
If $\intr(K)\cap L=\varnothing$, then $\dim (K\cap L)\leq 1$.
\end{lemma}

\begin{lemma}\label{projsc}
If $K$ is a smooth strictly convex proper cone and $H(u,0)$ is supporting
hyperplane to $K$ through a boundary point
point $x\not= 0$ of $K$, then $P^{-1}_K(K\cap H(u,0))= \spa \{x,u\}$
where $\spa M$ stands for the linear span of the set $M$.
Thus the set of points which projects by $P_K$ on the ray on
the boundary of $K$ engendered by $x$ is a two-dimensional subspace.
\end{lemma}
An example for a smooth strictly convex proper cone
is the so called Lorentz or ice cream cone:

The \emph{Lorentz or ice cream cone} $L\subset\R^m\times\R$ is defined by $$L=\{(x,t)\in\R^m\times\R:t\ge\|x\|\}.$$ 
It is a self-dual, smooth strictly convex cone. Other examples of self-dual, smooth, strictly convex cones
can be found in \cite{Iochum1984}.


\section{Preliminary results}

We will use in the following proofs the following simplified form of Moreau's decomposition theorem \cite{Moreau1962}:

\begin{theorem}\label{moreau}
        Let $K$ be a closed convex cone in $\R^m$ and $K^*$ its dual. For any $x$ in $\R^m$ we 
	have $x=P_Kx-P_{K^*}(-x)$ and $\lng P_Kx,P_{K^*}(-x)\rng=0$. The relation $P_Kx=0$ 
	holds if and only if $x\in -K^*$.
\end{theorem}

One of the basic tools in our proofs is the following result
which can be derived from Theorem 1, Theorem 2 and Lemma 4 in \cite{NemethNemeth2013}:

\begin{theorem}\label{elso}
A closed convex set $C\subset\R^m$ with nonempty interior is $K$-isotone
if and only if it can be represented in the form
\begin{equation}\label{vegso}
C=\cap_{i\in \N} H_-(u_i,a_i),
\end{equation}
where each hyperplane $H(u_i,a_i)$ is tangent to $C$ and is $K$-isotone.
\end{theorem}

The next theorem follows from Theorem 1 and Lemma 4 of the above cited paper.

\begin{theorem}\label{masodik} 
If $C$ is a closed convex set, then it is $K$-isotone
if and only if it is $K^*$-isotone.
\end{theorem}


\section{Isotone projection onto a proper cone}

The following theorem can be considered a main result of our note, which serves also
as basic tool for the next results.

\begin{theorem}\label{pc}
	Let $K,L$ be proper cones. If $K$ is an $L$-isotone projection set and $\intr(K^*)\cap L$ or $\intr(K^*)\cap L^*$ is nonempty, then $K$ is subdual 
	and $K\subset L\subset K^*$. 
\end{theorem}

\begin{proof}
Suppose first that $\intr(K^*)\cap L\ne\varnothing$ and let $v\in\intr(K^*)\cap L$. Then, $-v\in\intr(-K^*)$.  Consider an arbitrary 
	element $u\in K$ and an arbitrary positive integer $n$. Then, $(1/n)u-v\in -K^*\iff u-nv\in -K^*$ if $n$ is large enough. By using
	$u-nv\le_L u$, the $L$-isotonicity of $P_K$ and Theorem \ref{moreau}, we get $0=P_K(u-nv)\le_L P_K(u)=u$. Thus $K\subset L$.
	
	Hence $L^*\subset K^*$, or equivalently $K^*\cap L^*=L^*$ which, by using that $L$ is proper (and hence $L^*$ as well), implies 
	$\varnothing\ne\intr(L^*)=\intr(K^*\cap L^*)=\intr(K^*)\cap\intr(L^*)\subset\intr(K^*)\cap L^*$. Since $K$ is an $L$-isotone projection set,
	from Theorem \ref{masodik} it follows that $K$ is also an
	 $L^*$-isotone projection set. Since $\intr(K^*)\cap L^*\ne\varnothing$, we can use the above reasonings to
	get  $K\subset L^*$, which implies $L\subset K^*$. Hence, $K\subset L\subset K^*$ which also shows that $K$ is subdual. 
	
	Next suppose
	that $\intr(K^*)\cap L^*\ne\varnothing$. Then, by using that $K$ is an $L^*$-isotone projection set and the above result
	with $L^*$ replacing $L$, we get that $K$ is subdual and $K\subset L^*\subset K^*$, which implies $K\subset L\subset K^*$.
\end{proof}


\section{The case of self-dual $K$}

We remember that the proper cone $K\subset \R^m_+$ is called \emph{isotone projection cone} if it is $K$-isotone.
A direct verification shows that $\R^m_+$ is an isotone projection cone.
\begin{corollary}\label{csdi}
	Let $K$ be a self-dual cone and $L$ a proper cone. If $K$ is an $L$-isotone projection set and $\intr(K)\cap L$ or $\intr(K)\cap L^*$ is nonempty, 
	then $K=A\R^m_+$, for some orthogonal matrix $A$. Accordingly, the only proper cones $L$ such that $\R^m_+$ is $L$-isotone are the
	orthants of the reference system.
\end{corollary}

\begin{proof}
	By Theorem \ref{pc}, we get $K=L$. Thus by the main result in \cite{IsacNemeth1986}, $K$ is a self-dual isotone projection cone, or equivalently 
	$K=A\R^m_+$, for some orthogonal matrix $A$. 
	
	Suppose now that $\R^m_+$ is $L$-isotone with $L$ proper. Denote by $K$
	an orthant of the reference system, with $\intr (K)\cap L\not=\emptyset$.
	$K$ is self-dual, hence we must have by Theorem \ref{pc} that $K\subset L$ since $K$ is $L$-isotone
	together with $\R^m_+$ (this follows from Theorem \ref{elso}, $K$ and $\R^m_+$ having the same supporting hyperplanes).
	Now $K$ is also self-dual, hence $L=K$ by the first part of our proof.
\end{proof}

Denote by $\p$ the boundary mapping of sets.

\begin{proposition}\label{plorno}
	If $K$ is a sefdual, smooth, strictly convex cone in $\R^m$ with $m\geq 3$
	, then there is no proper cone
	$L$ in $\R^m$ such that $K$ is an $L$-isotone projection set.
\end{proposition}

\begin{proof}
	Suppose to the contrary, that $L\subset\R^m$ is a proper cone such that $K$ is a $L$-isotone projection set.
	
	Let us first assume that 
	$\intr(K)\cap L\ne\varnothing$. 
	Then, by using that $K$ is self-dual and Corollary \ref{csdi}, we get that 
	$K=A\R^m_+$ for some orthogonal matrix $A$, which is absurd because $K$ is not polyhedral.

	Next, assume that 
	\begin{equation}\label{empty}
	\intr(K)\cap L=\varnothing. 
	\end{equation}
	Since $K$ is a $L$-isotone projection set, we have that $P_K(L)\subset K\cap L$. Since 
	$\intr(K)\cap L=\varnothing$, we have that 
	\begin{equation}\label{plk}
		P_K(L)\subset \p K\cap L\subset K\cap L. 
	\end{equation}
	
	We show first that $P_K(L)\not= \{0\}.$
	To this end we observe that since $K$ is $L$-isotone, so is $-K$ (\cite{NemethNemeth2012a} Lemma 3). 
	The assumption $\intr(-K)\cap L\not=\varnothing$ would yield a contradiction as at the beginning of our proof.
	Hence
	$$L\subset \R^m \setminus (\intr(K)\cup \intr(-K)).$$
	Since $K$ is self-dual, $-K= P^{-1}_K(\{0\})$ by Theorem \ref{moreau}.
	Now, $L$ being proper, it must have points in $\R^m\setminus (K\cup -K)$, which 
	confirms our claim. 
	
	We must have according to (\ref{empty}) and Lemma \ref{intersec} that $K\cap L$ is an
	one-dimensional ray on the boundary of $K$ and $P_K(L)$ is itself this ray.
	Now, $L\subset P^{-1}_K(P_K(L))$ is contained by Lemma \ref{projsc} in a two-dimensional subspace.
	Hence $L$ cannot be a proper cone.
	
	\end{proof}
	

\section{Isotone projection onto a simplicial cone}

Let $e_1,\dots,e_m\in\R^m$ be linearly independent and $K=\cone\{e_1,\dots,e_m\}$ be a simplicial cone. Let
$\mc E=\{x=(x^1,\dots,x^m)^\top\in\R^m:|x^i|=1,\textrm{ }i=1,\dots m\}$ and $\varepsilon\in\mc E$. 
Denote \[K_\varepsilon=\cone\{\varepsilon^1e_1,\dots,\varepsilon^me_m\}.\] 

\begin{proposition}\label{pepsi}
	Let $K\subset\R^m$ be a simplicial cone and $L$ a proper cone such that $K$ is an $L$-isotone projection set. Then, there 
	exists an $\varepsilon\in\mc E$ such that $K_\varepsilon$ is subdual, $L$-isotone and $K_\varepsilon\subset L\subset K^*_\varepsilon$.
\end{proposition}

\begin{proof}
	Since $\cup_{\varepsilon\in\mc E} K^*_\varepsilon=\R^m$ and $L$ is proper we have that $\intr(K^*_\varepsilon)\cap L\ne\varnothing$
	for some $\varepsilon\in\mc  E$. Since the tangent hyperplanes of $K_\varepsilon$ coincide with the tangent hyperplanes of $K$
	it follows from Theorem \ref{elso}
	that $K_\varepsilon$ is also an $L$-isotone projection set.  Hence, the result follows from Theorem \ref{pc}. 
\end{proof}

Denote $N=\{1,\dots,n\}$. For an index set $I\subset N$ denote $I^c=N\setminus I$ the complementary index set of $I$. For any vector $x\in\R^m$
denote by $\diag(x)$ the diagonal matrix which contains $x$ in the main diagonal such that the $(i,i)$-th entry of $\diag(x)$ is $x^i$, (where
any vector $y\in\R^m$ is written as $y=(y^1,\dots,y^m)^\top$), while its other entries are $0$. A simplicial cone $K=\cone\{e_1,\dots,e_m\}$ is subdual if and only if $E^\top E$
is an $m\times m$ nonnegative matrix, where $E=(e_1,\dots,e_m)$ (the matrix with columns $e_i$). $E$ is called the matrix of $K$.

\begin{lemma}\label{lepsi}
	Let $K=\cone\{e_1,\dots,e_m\}$ be a simplicial cone. Then, there exists a $\varepsilon\in\mc E$ such that $K_\varepsilon$ is subdual if
	and only if there exists an index set $I\subset N$ such that $\lng e_i,e_j\rng\ge0$ for any $i,j\in I$, $\lng e_k,e_\ell\rng\ge0$ for any 
	$k,\ell\in I^c$, and $\lng e_i,e_k\rng\le0$ for any $i\in I$ and any $k\in I^c$.
\end{lemma}

\begin{proof}
	Let $\varepsilon\in\mc E$. Let $I=\{i\in N:\varepsilon_i=1\}$. Then, $I^c=\{i\in N:\varepsilon_i=-1\}$. 
	Then, the matrix of $K_\varepsilon$ is $ED$, where $E=(e_1,\dots,e_m)$ and $D=\diag(\varepsilon)$. Then, $K_\varepsilon$ is subdual
	if and only if $DE^\top ED=(ED)^\top ED$ is nonnegative. However, $DE^\top ED$ is the matrix whose rows and columns corresponding to
	each index of the index set $I^c$ are the corresponding rows and columns of $E^\top E$, respectively multiplied by $-1$. Hence, 
	$DE^\top ED$ is nonnegative if and only if 
	$\lng e_i,e_j\rng\ge0$ for any $i,j\in I$, $\lng e_k,e_\ell\rng\ge0$ for any $k,\ell\in I^c$, and $\lng e_i,e_k\rng\le0$ for any 
	$i\in I$ and any $k\in I^c$. This follows because the $(r,s)$-th entry of $E^\top E$ is $\lng e_r,e_s\rng$ for any $r,s\in N$ and 
	therefore for any $i,j\in I$ the $(i,j)$-th entry of $DE^\top ED$ is $\lng e_i,e_j\rng$, for any $k,\ell\in I^c$ the $(k,\ell)$-th entry
	of $DE^\top ED$ is $\lng e_k,e_\ell\rng$, and for any $i\in I$ and any $k\in I^c$ the $(i,k)$-th entry of $DE^\top ED$ is 
	$-\lng e_i,e_k\rng$.
\end{proof}

\begin{proposition}\label{psubd}
	Let $K=\cone\{e_1,\dots,e_m\}\subset\R^m$ be a simplicial cone and $L$ a proper cone such that $K$ is an $L$-isotone projection set. 
	Then, there exists an index set $I\subset N$ such that $\lng e_i,e_j\rng\ge0$ for any $i,j\in I$, $\lng e_k,e_\ell\rng\ge0$ for any 
	$k,\ell\in I^c$, and $\lng e_i,e_k\rng\le0$ for any $i\in I$ and any $k\in I^c$.
\end{proposition}

\begin{proof}
	It follows from Proposition \ref{pepsi} and Lemma \ref{lepsi}.
\end{proof}

\begin{corollary}\label{ellpeld}
	Let $K=\cone\{e_1,\dots,e_m\}\subset\R^m$ be a simplicial cone. Suppose that $i,j,k\in N$ are three pairwise distinct indices such that
	$\lng e_i,e_j\rng<0$, $\lng e_i,e_k\rng<0$ and $\lng e_j,e_k\rng<0$. Then there is no proper cone $L$ such that $K$ is an $L$-isotone
	projection set.
\end{corollary}

\begin{proof}
	Suppose that $L$ is a proper cone such that $K$ is an $L$-isotone projection set. From Proposition \ref{psubd}, there exists an index set
	$I\subset N$ such that one of $i,j$ belong to $I$ and another one to $I^c$, and such that similar statements hold for $i,k$ and 
	$j,k$, respectively. 
	This leads to an obvious contradiction. Hence, there is no proper cone $L$ such that $K$ is an $L$-isotone projection set.
\end{proof}

\begin{proposition}
	Let $K\subset\R^m$ be an isotone projection cone. Then, $K_\varepsilon$ is a $K$-isotone projection set for any $\varepsilon\in E$.
\end{proposition}

\begin{proof}
	Since $K$ is a $K$-isotone projection set and the tangent hyperplanes of $K_\varepsilon$ coincide with the tangent hyperplanes of $K$,
	from Theorem \ref{elso} it 
	follows that $K_\varepsilon$ is also a $K$-isotone projection set.
\end{proof}

\begin{remark}
From this proposition it follows that for $K$ an isotone projection cone 
each member of the family $\{K_\varepsilon: \varepsilon \in \mc E\}$	is a $K$-isotone
simplicial cone. Obviously, $\intr (K_\varepsilon) \cap K=\varnothing$ whenever $\diag \varepsilon$
is not the identity matrix. Hence by Theorem \ref{pc} in this case we must also have 
$$ \intr(K^*_\varepsilon) \cap K=\varnothing,\;\; \textrm{and}\;\; \intr (K^*_\varepsilon) \cap K^*=\varnothing.$$
\end{remark}


\section{The case of $\R^m_+$-isotone projection cones}

To show that in contrast with Corollary \ref{ellpeld} there exists a large
class of cones which can be $\R^m_+$-isotone projection cones or more
general polyhedral cones for which there are order relations
with respect to which they admit isotone projections, we cite
Theorem 3 in \cite{NemethNemeth2016} (see \cite{NemethNemeth2016} for the definition of a facet):

\begin{theorem}\label{orthisosubcone}

If $K$ is a generating closed convex
cone in  $\R^m$, then it is $\R^m_+$-isotone,
if and only if it is a polyhedral cone of the form
\begin{equation}\label{k1}
K= \cap_{k<l} (H_-(a_{kl1},0)\cap H_-(a_{kl2},0)), \;\;k,\,l\in \{1,\dots,m\}
\end{equation}
where $a_{kli}$ are nonzero vectors with $a_{kli}^ka_{kli}^l\leq 0$ and $a_{kli}^j =0$ for $j\notin \{k,l\},\;i=1,2.$
Hence $K$ possesses at most $m(m-1)$ facets. There exists a cone $K$ of the above form with exactly $m(m-1)$ facets.

\end{theorem}

We remark that in this theorem the cone $K$ may be a proper or only a closed and convex generating cone.

\begin{remark}

As a family of simplicial subcones contained in $\R^m_+$ which are
$\R^m_+$-isotone we mention the family of the so called istonic regression 
cones, among which the single cone which is itself an isotone projection cone too is the
monotone nonnegative cone (see Corollary 1 and 2 in \cite{NemethNemeth2016} and the definitions therein).

\end{remark}

\begin{corollary}

If $K$ is an $\R^m_+$-isotone proper cone, then exactly one of the alternatives
\begin{enumerate}
\item $K\subset \R^m_+$,
\item $\intr(K^*)\cap \R^m_+=\varnothing$
\end{enumerate}
holds.
\end{corollary} 

\begin{proof}
If item 1 holds, then $$\R^m_+=(\R^m_+)^*\subset K^*,$$ and hence item 2 does not hold. If item 2 does not hold, then by Theorem \ref{pc} item 1
holds.  
\end{proof}

\bibliographystyle{abbrv}
\bibliography{ordiso}

\end{document}